\documentclass[a4paper,11pt]{article}
\usepackage[utf8]{inputenc}
\usepackage{amsthm,amsmath,amssymb,anysize,enumerate,tikz-cd,dsfont}
\usepackage[mathscr]{eucal}
\usepackage[cmtip,all]{xy}
\usepackage{hyperref}
\usepackage{cite}
\marginsize{2cm}{2cm}{2cm}{2cm}

\renewcommand{\AA}{\mathds A}
\newcommand{\CC}{\mathds C}

\newcommand{\ZZ}{\mathds Z}

\newcommand{\cB}{\mathcal B}

\newcommand{\cD}{\mathcal D}

\newcommand{\cM}{\mathcal M}

\newcommand{\cO}{\mathcal O}

\newcommand{\cR}{\mathcal R}

\newcommand{\cV}{\mathcal V}

\newcommand{\fa}{\mathfrak{a}}
\newcommand{\fb}{\mathfrak{b}}

\newcommand{\dd}{\partial}
\newcommand{\btil}{\tilde{b}}

\newcommand{\wt}{\widetilde}
\newcommand{\ux}{\underline{x}}
\newcommand{\uz}{\underline{z}}

\renewcommand{\le}{\langle}
\newcommand{\re}{\rangle}

\newcommand{\Gr}{\operatorname{Gr}}

\newcommand{\lcm}{\operatorname{lcm}}
\newcommand{\Sing}{\operatorname{Sing}}

\newcommand\blfootnote[1]{
  \begingroup
  \renewcommand\thefootnote{}\footnote{#1}
  \addtocounter{footnote}{-1}
  \endgroup
}

{\theoremstyle{definition}
\newtheorem{defi}{Definition}[section]
}
{\theoremstyle{remark}
\newtheorem{rem}[defi]{Remark}

\newtheorem{exas}[defi]{Examples}}
\newtheorem{thm}[defi]{Theorem}
\newtheorem{prop}[defi]{Proposition}

\begin{document}

\title{On the reduced Bernstein-Sato polynomial of Thom-Sebastiani singularities}
\author{Alberto Casta\~no Dom\'{\i}nguez\footnote{\noindent Corresponding author. ORCID number 0000-0002-3473-6724. Departamento de Álgebra, Universidad de Sevilla, C/Tarfia s/n, 41012 Sevilla (Spain)\\
albertocd@us.es}~~and Luis Narv\'aez Macarro\footnote{\noindent ORCID number 0000-0003-4316-5019. Departamento de Álgebra and Instituto de Matem\'aticas (IMUS), Universidad de Sevilla, C/Tarfia s/n, 41012 Sevilla (Spain)\\
narvaez@us.es}}
\date{}

\maketitle

\begin{abstract}
Given two holomorphic functions $f$ and $g$ defined in two respective germs of complex analytic manifolds $(X,x)$ and $(Y,y)$, we know thanks to M. Saito that, as long as one of them is Euler homogeneous, the reduced (or microlocal) Bernstein-Sato polynomial of the Thom-Sebastiani sum $f+g$ can be expressed in terms of those of $f$ and $g$. In this note we give a purely algebraic proof of a similar relation between the whole functional equations that can be applied to any setting (not necessarily analytic) in which Bernstein-Sato polynomials can be defined.
\blfootnote{
The authors are thankful to the anonymous referees for their valuable suggestions and remarks.\\
The authors are partially supported by PID2020-114613GB-I00 and P20\_01056. The first author is also partially supported by VI PPIT US-2018-II.5. They have no further relevant interests to disclose.}
\end{abstract}

\section{Introduction}

Let $(X,x)$ and $(Y,y)$ be two germs of complex analytic manifolds of respective dimensions $n$ and $m$. We will consider two nonzero holomorphic functions $f\in\cO_{X,x}$ and $g\in\cO_{Y,y}$, not necessarily reduced, and their Thom-Sebastiani sum $h:=f+g\in\cO_{X\times Y,(x,y)}$.

Let $s$ be a dummy variable. The Bernstein module of $f$ is the $\cD_{X,x}[s]$-module $\cB_f=\cO_{X,x}[s,f^{-1}]\cdot f^s$, with the usual action of $\cD_{X,x}$. The Bernstein-Sato polynomial of $f$ is then the monic generator of the ideal of polynomials $b(s)\in\CC[s]$ verifying that
$$ P(s)f\cdot f^s=b(s)\cdot f^s $$
in $\cB_f$, for some $P(s)\in\cD_{X,x}[s]$, or equivalently, the minimal polynomial of the action of $s$ on the quotient module $\cD_{X,x}[s]\cdot f^s/\cD_{X,x}[s]\le f\re\cdot f^s$. We will denote it by $b_f(s)$.

As long as $f$ is not invertible, it is well known that $s+1$ divides $b_f(s)$, so we can define the \textit{reduced} Bernstein-Sato polynomial of $f$ as $\btil_f(s):=b_f(s)/(s+1)$, which is also the minimal polynomial of the action of $s$ on the $\cD_{X,x}[s]$-module $\cM_f:=\cD_{X,x}[s]\cdot f^s/\cD_{X,x}[s]J_f\cdot f^s$ (see, for instance, \cite[Lemma 1.1]{Granger} and the commentary thereafter), where $J_f$ is the ``true'' Jacobian ideal of $f$; in local coordinates $x_1,\ldots,x_n$, we have $J_f:=\le f,f'_{x_1},\ldots,f'_{x_n}\re\subseteq\cO_{X,x}$. In that sense, we have a new functional equation of the form
$$P(s)\cdot f^s=\btil_f(s)\cdot f^s$$
in $\cB_f$, where now $P(s)$ belongs to $\cD_{X,x}[s]J_f$.

The polynomial $\btil_f(s)$ is also called \textit{microlocal} Bernstein-Sato polynomial, and we will see the reason later on in Section \ref{Sec:Alternatives}.

Everything in the paragraphs above can be analogously defined for $g$ and $h$, and thus we may wonder about the relation between $\btil_f$, $\btil_g$ and $\btil_h$. Before continuing, let us state some notation and define an important notion.

Given a polynomial $p(s)\in\CC[s]$, we will denote by $R_p\subseteq\CC$ the set of the opposites of its roots. For any given $\alpha\in R_p$, we will call $m_\alpha(p)$ its multiplicity as root of $p$.

\begin{defi}\label{defi:ConvPoli}
Let $a(s),b(s)\in\CC[s]$ be two nonzero polynomials. Let $(a*b)(s)=a(s)*b(s)\in\CC[s]$ be the monic polynomial with roots $R_{a*b}=R_a+R_b$ and multiplicities
$$m_\gamma(a*b)=\max\{m_\alpha(a)+m_\beta(b)-1\,:\,\alpha+\beta=\gamma\},$$
for every $\gamma\in R_{a*b}$. We will call $a*b$ the \textit{star operation} of $a$ and $b$.
\end{defi}

We will use the convention that adding the empty set to any other one gives the empty set. Therefore, if, for example, $f$ defines a smooth divisor, in such a way that $\btil_f=1$, then $\btil_f*\btil_g=1=\btil_h$ as well.

In \cite{Saito}, M. Saito studied the existing relation between the reduced Bernstein-Sato polynomial of $h$ and those of $f$ and $g$ and proved the following ([\textit{loc. cit.}, Proposition 0.7, Theorem 0.8]):

\begin{thm}\label{thm:Saito}
Under the same conditions as above,
\begin{itemize}
  \item $R_{\btil_f}+R_{\btil_g}\subseteq R_{\btil_h}+\ZZ_{\leq0}$ and $R_{\btil_h}\subseteq R_{\btil_f}+R_{\btil_g}+\ZZ_{\geq0}$.
  \item In addition, if there exists a germ of vector field $\chi\in\Theta_{Y,y}$ such that $\chi(g)=g$, then $(\btil_f*\btil_g)(s)=\btil_h(s)$.
\end{itemize}
\end{thm}

\begin{rem}
The condition on $g$ given in the second point is usually referred to as being Euler-homogeneous at $y$, $\chi$ being an Euler field for $g$. Two easy consequences of that fact are that the Jacobian ideal $J_h\subset\cO_{X\times Y,(x,y)}$ is just the sum of the extended Jacobian ideals $J_f^e+J_g^e$, and that $t\cdot g^t=\chi\cdot g^t$ in $\cO_{Y,y}[t,g^{-1}]\cdot g^t$.
\end{rem}

Saito's proof of the theorem uses the power of the Kashiwara-Malgrange filtration on $\cD_{X,x}[t,\dd_t]$. This note is the result of our efforts to find a purely algebraic proof of such result, that can be extended to a more general context. What is new, to the best of our knowledge, is an explicit expression for the functional equation for the reduced Bernstein-Sato polynomial of the sum $h=f+g$ in terms of those for $f$ and $g$:

\begin{thm}\label{thm:DivisionPrevio}
Let $(X,x)$ and $(Y,y)$ be two germs of complex analytic manifolds of respective dimensions $n$ and $m$. Let $f\in\cO_{X,x}$ and $g\in\cO_{Y,y}$ two nonzero holomorphic functions, and let $h:=f+g\in\cO_{X\times Y,(x,y)}$. Assume moreover that $\chi\in\Theta_{Y,y}$ is an Euler vector field for $g$, i.e. $\chi(g)=g$, and that we have functional equations:
\begin{equation*}
\begin{split}
   P(s)\cdot f^s=\btil_f(s)\cdot f^s \,\text{ in } \cO_{X,x}[s,f^{-1}]\cdot f^s, & \text{ with } P(s)\in\cD_{X,x}[s]J_f,\\
   Q\cdot g^t=\btil_g(t)\cdot g^t \,\text{ in } \cO_{Y,y}[t,g^{-1}]\cdot g^t,  & \text{ with } Q\in\cD_{Y,y} J_g.
\end{split}
\end{equation*}
Then, we have the functional equation
$$R(u)\cdot h^u=(\btil_f*\btil_g)(u)\cdot h^u$$
in $\cO_{X\times Y,(x,y)}[u,h^{-1}]\cdot h^u$, where $R(u)=P(u-\chi)A(u,\chi)+B(u,\chi)Q\in \cD_{X\times Y,(x,y)}J_h$. There, $A(s,t)$ and $B(s,t)$ are certain polynomials in $\CC[s,t]$ that can be obtained from $\btil_f$ and $\btil_g$, whose meaning will be explained at the end of subsection \ref{subsec:Conv}.

In particular, $\btil_h$ divides $\btil_f*\btil_g$.
\end{thm}

Note that in the functional equation for $g$ the operator $Q$ does not depend on $t$. This is because $g$ being Euler-homogeneous implies that $t\cdot g^t=\chi\cdot g^t$.

Again, notice that we do not just prove that one polynomial divides the other, but provide a concrete functional equation for $\btil_f*\btil_g$. The statement on just the divisibility was first proved by Yano in \cite[Theorem 3.15]{Yano} in the particular case that $g$ is quasi-homogeneous and has an isolated singularity at $y$ (note that $\btil_f*\btil_g$ is hidden in the statement due to the simple expression of $\btil_g$).

In fact, even though the statements of Theorems \ref{thm:Saito} and \ref{thm:DivisionPrevio} above relate holomorphic functions on germs of complex manifolds, the functional equation and thus the relation between the reduced Bernstein-Sato polynomials in the case $g$ is Euler-homogeneous can be easily generalized to the global algebraic and formal cases. The latter is just a consequence of the extension $\cO_{X,x}\rightarrow\CC[[x_1,\ldots,x_n]]$ being faithfully flat for a choice of local parameters $x_1,\ldots,x_n$ at $x$.

Let us consider with a little more detail the first case, so assume $f$ and $g$ are nonzero polynomials. Then, denoting by $\cV(p)\subseteq\AA_\CC^r$ the vanishing locus of a polynomial $p\in\CC[x_1,\ldots,x_r]$, we know that their associated (algebraic) Bernstein-Sato polynomials are just the least common multiples of their local versions at each point of $\cV(f)\subseteq\AA_\CC^n$ and $\cV(g)\subseteq\AA_\CC^m$, respectively (see, for instance, \cite[Proposition 4.2.1]{LuisMebkhout}). In fact we could consider only the respective singular points, for the reduced Bernstein-Sato polynomials are $1$ otherwise. In that case, let us write $\btil_{h,(x,y)}(s)$ to denote the local polynomial at the point $(x,y)\in \AA_\CC^n\times \AA_\CC^m$. Therefore, $\btil_h(s)=\lcm\{\btil_{h,(x,y)}(s)\,:\,(x,y)\in\Sing\cV(h)\subseteq \AA_\CC^n\times \AA_\CC^m\}$. The variety $\Sing\cV(h)$ is given by the equations
\begin{equation}\label{eq:singh}
h=0,\quad h'_{x_i}=0,\quad h'_{y_j}=0,
\end{equation}
for $i=1,\ldots,n$, $j=1,\ldots,m$. Since $h'_{x_i}=f'_{x_i}$ and $h'_{y_j}=g'_{y_j}$ and $g$ is Euler-homogeneous, the vanishing of the $g'_{y_j}$ implies that of $g$, so the equations \eqref{eq:singh} define the same set as $\Sing\cV(f)\times\Sing\cV(g)$. In conclusion, the functional equation for $\btil_f*\btil_g$ at each point of $\AA_\CC^n\times \AA_\CC^m$ implies the same relation for its global versions.

In fact, the proof of Theorem \ref{thm:DivisionPrevio} can be extended almost literally to any context in which we have a properly working formal functional equation, like differentially admissible algebras (see \cite[Definition 1.2.3.6, Theorem 3.2.2.1]{Luis} and \cite[Hypothesis 2.3, Proposition 3.10]{Nunez}), nonregular algebras or direct summands (\hspace{-.1pt}\cite[Proposition 2.18, Theorem 3.24]{Josepycia}).

Regarding Bernstein-Sato polynomials of ideals (see \cite{BMS}), we know thanks to \cite[Theorem 1.1]{Mus} that the Bernstein-Sato polynomial of a nonzero ideal $\fa=\le f_1(\ux),\ldots,f_r(\ux)\re\subseteq\CC[x_1,\ldots,x_n]$ is exactly the reduced Bernstein-Sato polynomial of $z_1f_1(\ux)+\ldots+z_rf_r(\ux)\in\CC[\ux,\uz]$. Therefore, since such a polynomial is always Euler homogeneous at the origin of $\AA_\CC^n\times\AA_\CC^r$ (assuming $0\in\cV(f_1,\ldots,f_r)\subseteq\AA_\CC^n$), the Bernstein-Sato polynomial of the sum of two ideals $\fa\subseteq\CC[x_1,\ldots,x_n]$ and $\fb\subseteq\CC[y_1,\ldots,y_m]$ always divides $b_\fa*b_\fb$.

On the other hand, we believe it would be worthwhile to find a proof for the remaining divisibility $\btil_f*\btil_g|\btil_h$ that is as formal or algebraic as possible, following the spirit of the proof of the theorem above. However, up to now we have not been able to do it.

The rest of this note is organised as follows: in Section 2, we provide an alternative way to obtain the star operation of two polynomials and we give a proof that relates directly the definition of $\btil_f$ given here and M. Saito's one using a microlocal construction. Finally, in section 3 we prove Theorem \ref{thm:DivisionPrevio}.

\section{Alternative definitions}\label{Sec:Alternatives}

In this section we will give a couple of equivalent definitions for both the star operation of two polynomials and the reduced Bernstein-Sato polynomial. Their actual equivalence might be folklore, but we have not been able to find it in the existing literature.

\subsection{Two operations with polynomials}\label{subsec:Conv}

Let us recall Definition \ref{defi:ConvPoli}: the star operation of two polynomials $a,b\in\CC[s]$ is another polynomial $(a*b)(s)\in\CC[s]$, such that $R_{a*b}=R_a+R_b$ and, for any $\gamma\in R_{a*b}$, its multiplicity is given by the highest value of $m_\alpha(a)+m_\beta(b)-1$, where $\alpha\in R_a$, $\beta\in R_b$ and $\alpha+\beta=\gamma$.

We will use another approach to work with such operation:

\begin{prop}\label{prop:ConvEquiv}
Let $a,b\in\CC[s]$ be two nonzero polynomials and let us denote by $(a\bullet b)(s)\in\CC[s]$ the monic polynomial that verifies that $\le a(s),b(t)\re\cap\CC[s+t]=\le (a\bullet b) (s+t)\re$. Then, $a\bullet b=a*b$.
\end{prop}

Note that, in the statement of the proposition, we can take $(a\bullet b)(s)$ to be the generator of the ideal $\le a(s-t),b(t)\re\cap\CC[s]$, just by a simple change of variables. This definition will be useful later on.

\begin{proof}
The proof is elementary but a bit long. If any of $a$ or $b$ is constant there is nothing to show. Therefore, let us prove first the proposition when $a(s)=(s-\alpha)^d$ and $b(s)=(s-\beta)^e$, for some $\alpha,\beta\in \CC$ and $d,e\geq1$. In that case, clearly $(a*b)(u)=(u-\alpha-\beta)^{d+e-1}$. Let us consider then the ideal $I=\le (s-\alpha)^d,(t-\beta)^e\re\cap\CC[s+t]$. Expanding $(s+t-\alpha-\beta)^{d+e-1}=((s-\alpha)+(t-\beta))^{d+e-1}$ makes clear that $(a*b)(s+t)$ belongs to $I$ and is a multiple of $(a\bullet b)(s+t)$.

To see the converse, we can assume, up to a simple change of variables, that $\alpha=\beta=0$ for the sake of simplicity. Consider any $p(s+t)=\sum_{i=0}^N p_i(s+t)^i\in I$. If $N\geq d+e-1$, reasoning as above we can claim that $p\in I$ if and only if $\tilde p:=\sum_{i=0}^{d+e-2} p_i(s+t)^i$ lies within $I$ too, but that implies that $p_i=0$ for each $i=0,\ldots,d+e-2$. Indeed, modulo $s^d$ and $t^e$, the only nonvanishing term of degree $d+e-2$ is $p_{d+e-2}\binom{d+e-2}{d-1}$, that must be zero if $\tilde p\in I$. We can continue the same argument with the remaining coefficients. Therefore, $\tilde p=0$ and $I=\le (a*b)(s+t)\re$, that is, $a\bullet b=a*b$.

Let us prove now that, if $a(s),b(s),q(s)\in\CC[s]$, then
\begin{equation}\label{eq:FormulaMCM}
\lcm(a,b)\star q=\lcm(a\star q,b\star q),
\end{equation}
where $\star=*,\bullet$. Since both operations are commutative in $\CC[s]$, that suffices to finish the proof. We will write $c(s):=\lcm(a(s),b(s))$ for the sake of brevity.

First, let us take $\star=*$. On one hand, $R_{c}=R_a\cup R_b$, so $R_{c}+R_q=(R_a+R_q)\cup (R_b\cup R_q)$. Since $m_\alpha(c)=\max\{m_\alpha(a),m_\alpha(b)\}$ for every $\alpha\in R_{c}$, we have that for any $\gamma\in R_{c}+R_q$,
\begin{equation}\begin{split}\label{eq:MultMCM}
m_\gamma(c*q)&=\max\{m_\alpha(c)+m_\beta(q)-1\,:\,\alpha+\beta=\gamma\}\\
&=\max\big\{\max\{m_\alpha(a)+m_\beta(q)-1\,:\,\alpha+\beta=\gamma\}, \max\{m_\alpha(b)+m_\beta(q)-1\,:\,\alpha+\beta=\gamma\}\big\}.
\end{split}\end{equation}
On the other hand, the opposites of the roots of $\lcm(a*q,b*q)$ are $R_{a*q}\cup R_{b*q}=(R_a+R_q)\cup(R_b+R_q)$ and their multiplicities are exactly the second line of formula \eqref{eq:MultMCM}, so formula \eqref{eq:FormulaMCM} holds.

Let $\star$ be $\bullet$ now, and let us show that $I:=\le c(s),q(t)\re=J:=\le a(s),q(t)\re\cap\le b(s),q(t)\re\subseteq\CC[s,t]$. It is clear that $I\subseteq J$; let us show the reverse inclusion. To do so, let us also write $d(s):=\gcd(a(s),b(s))$, so that we have a B\'ezout identity of the form $d=\alpha a+\beta b$ in $\CC[s]$. Then, if we have a polynomial $p(s,t)=m(s,t)a(s)+n(s,t)q(t)=m'(s,t)b(s)+n'(s,t)q(t)\in J$, we can also write $dp=\alpha ap+\beta bp$ and use both representations of $p$ as an element of $J$, such that
$$p=\alpha\frac{a}{d}p+\beta\frac{b}{d}p=\alpha m'\frac{ab}{d}+\alpha \frac{a}{d} n'q+\beta m \frac{ab}{d}+\beta\frac{b}{d}nq\in I,$$
since $c=ab/d$.

In order to finish, just note that the generator of $I\cap\CC[s+t]$ is $\lcm(a,b)\bullet q$, whereas the generator of $J\cap\CC[s+t]$ is $\lcm(a\bullet q,b\bullet q)$.
\end{proof}

Now we can explain all actors involved in the statement of Theorem \ref{thm:DivisionPrevio}. Namely, since we know that $(\btil_f*\btil_g)(s)$ is the generator of the ideal $\le \btil_f(s-t),\btil_g(t)\re\cap\CC[s]$, there must exist $A(s,t),B(s,t)\in\CC[s,t]$ such that $(\btil_f*\btil_g)(s)=A(s,t)\btil_f(s-t)+B(s,t)\btil_g(t)$. Those are the polynomials we use to build up the functional equation for $\btil_f*\btil_g$.

\subsection{Reduced Bernstein-Sato polynomial}

There are at least three known objects called the reduced Bernstein-Sato polynomial $\btil_f(s)$: the quotient of the usual polynomial by $s+1$, the obtained by the Jacobian approach noted in the introduction, that we will call ``Jacobian Bernstein-Sato polynomial'', and the microlocal Bernstein-Sato polynomial of M. Saito (see \cite[\S~1]{Saito}). Although it is well known that these last two ones provide the same object, we will include here a direct proof of the fact without showing that both of them are $b_f(s)/(s+1)$. Before that, let us comment on more about the microlocal setting, following [\textit{loc. cit.}].

Let us call $\wt\cB_f=\cO_{X,x}[\dd_t,\dd_t^{-1}]\cdot\delta(t-f)$, where $\delta(t-f)$ is a symbol representing the delta function supported on the graph $\{f=t\}$ on which $\cD_{X,x}$, $t$ and the integer powers of $\dd_t$ act in the usual way. Therefore, $\wt\cB_f$ can be endowed with the structure of $\wt\cD_t$-module, where by $\wt\cD_t$ we mean the ring $\cD_{X,x}[t,\dd_t,\dd_t^{-1}]$ (called $\wt\cR$ in Saito's construction).

We can define a $V$-filtration on $\wt\cD_t$ by setting $V^0\wt\cD_t=\cD_{X,x}[t\dd_t ,\dd_t^{-1}]$ and $V^p\wt\cD_t=\dd_t^{-p}V^0\wt\cD_t=V^0\wt\cD_t\dd_t^{-p}$. This filtration induces another one on $\wt\cB_f$ just by taking $G^p\wt\cB_f=V^p\wt\cD_t\cdot\wt\cB_f$. With all this in mind, the microlocal Bernstein-Sato polynomial $\btil_{f,m}(s)$ is defined as the minimal polynomial of the action of $s:=-\dd_tt$ on $\Gr_G^0\wt\cB_f$.

\begin{prop}
Let $f\in\cO_{X,x}$ be a nonzero holomorphic function, and let $\btil_{f,m}(s)$ be its microlocal Bernstein-Sato polynomial and $\btil_{f,J}(s)$ be its Jacobian Bernstein-Sato polynomial. Then, $\btil_{f,m}=\btil_{f,J}$.
\end{prop}
\begin{proof}
As we have said above, recall that $\btil_{f,m}(s)$ and $\btil_{f,J}(s)$ are, respectively, the minimal polynomials of the actions of $s$ on $\Gr_G^0\wt\cB_f$ and on $\cM_f=\cD_{X,x}[s]f^s/\cD_{X,x}[s]J_ff^s$, acting on the first object as $-\dd_tt$.

Following the well-known construction of \cite[\S~4]{Malgrange}, we can consider the isomorphism $\wt\cB_f\rightarrow\wt\cM:=\cD_{X,x}[t,\dd_t,\dd_t^{-1}]/\le t-f,\dd_i+f'_i\dd_t\re$ sending $\delta(t-f)$ to the generator $\bar{1}$ of $\wt\cM$. As a consequence, we know that $G^0\wt\cB_f=V^0\wt\cD_t\cdot\bar{1}= \cD_{X,x}[\dd_t t,\dd_t^{-1}]\cdot\bar{1}$ and $G^1\wt\cB_f=V^1\wt\cD_t\cdot\bar{1}=\dd_t^{-1}\cD_{X,x}[\dd_t t,\dd_t^{-1}]\cdot\bar{1}$.

Moreover, $\dd_t^{-1}\dd_i=-f'_i$ in $\wt\cM$, so we obtain that
$$\Gr_G^0\wt\cB_f\cong\frac{\cD_{X,x}[\dd_tt]+\cO_{X,x}[\dd_t^{-1}]_{>0}}{\cD_{X,x}[\dd_tt]J_f +\cO_{X,x}[\dd_t^{-1}]_{>0}} \cdot\bar{1},$$
where $\cO_{X,x}[\dd_t^{-1}]_{>0}$ represents the polynomials in $\dd_t^{-1}$ with coefficients in $\cO_{X,x}$ and no term of degree zero.

Now note that, as $\cO_{X,x}$-modules,
$$\frac{\cD_{X,x}[\dd_tt]+\cO_{X,x}[\dd_t^{-1}]_{>0}}{\cD_{X,x}[\dd_tt]J_f +\cO_{X,x}[\dd_t^{-1}]_{>0}}\cong\frac{\cD_{X,x}[\dd_tt]}{\cD_{X,x}[\dd_tt]J_f}$$
thanks to the second isomorphism theorem. Consequently,
$$\Gr_G^0\wt\cB_f\cong\frac{\cD_{X,x}[\dd_tt]}{\cD_{X,x}[\dd_tt]J_f}\cdot\bar{1}.$$
Finally, recall that there is an isomorphism $\cD_{X,x}[\dd_tt]\cdot\bar{1}\cong\cM_f$ induced by the isomorphism of $\cD_{X,x}$-modules $\cD_{X,x}[\dd_tt]\delta(t-f)\rightarrow\cD_{X,x}[s]\cdot f^s$ that sends $\dd_tt$ to $-s$. Using this last correspondence, we see that the minimal polynomial of $s$ on $\cM_f$ and of $-\dd_tt$ on $\wt\cB_f$ are the same.
\end{proof}

\section{Proof of the main result}

We provide in this section an elementary proof of Theorem \ref{thm:DivisionPrevio}.

\begin{thm}\label{thm:Division1}
Let $(X,x)$ and $(Y,y)$ be two nonzero germs of complex analytic manifolds of respective dimensions $n$ and $m$. Let $f\in\cO_{X,x}$ and $g\in\cO_{Y,y}$ two holomorphic functions, and let $h:=f+g\in\cO_{X\times Y,(x,y)}$. Assume moreover that $\chi\in\Theta_{Y,y}$ is an Euler vector field for $g$, i.e. $\chi(g)=g$, and that we have functional equations:
\begin{eqnarray*}
& \displaystyle b(s) \cdot f^s = P(s)\cdot  f^s \,\text{ in } \cO_{X,x}[s,f^{-1}]\cdot f^s \text{ with }\, P(s) = \sum_j P_j s^j,\ P_j \in \cD_{X,x} J_f,
&
\\
& \displaystyle c(t)\cdot  g^{t} = Q\cdot  g^{t} \,\text{ in } \cO_{Y,y}[t,g^{-1}]\cdot g^t \text{ with }\, Q  \in \cD_{Y,y} J_g = \cD_{Y,y} \langle g'_{y_1},\dots,g'_{y_m}\rangle,
\end{eqnarray*}
and let $A(s,t), B(s,t) \in \CC[s,t]$ be such that $(b * c)(s) = A(s,t) b(s-t) + B(s,t) c(t)$. Then, we have a functional equation:
$$   (b * c)(s)\cdot  h^s = R(s)\cdot  h^s \,\text{ in } \cO_{X\times Y,(x,y)}[s,h^{-1}]\cdot h^s \text{ with }\, R(s)=P(s-\chi) A(s,\chi) + B(s,\chi) Q,
$$
where $P(s-\chi) = \sum_{j} P_j (s-\chi)^j =  \sum_{j}  (s-\chi)^j P_j$. Moreover, $R(s) \in \cD_{X\times Y,(x,y)}[s] J_h$. In particular, $\btil_h$ divides $\btil_f*\btil_g$.
\end{thm}

\begin{proof} For any integer $k\geq1$, we obtain by expanding $h^k$ and $R(k)$ that
\begin{align}\begin{split}\label{eq:EcFuncLuis}
R(k) \big(h^k\big)&= R(k) \left(\sum_{\ell=0}^k \binom{k}{\ell} f^{k-\ell} g^\ell   \right)\\
&=\sum_{\ell=0}^k \binom{k}{\ell} \left(P(k-\chi) \Big( A(k,\chi) \big(  f^{k-\ell} g^\ell  \big) \Big) +
B(k,\chi)\Big( Q \big(f^{k-\ell} g^\ell  \big) \Big)\right).
\end{split}\end{align}
In the first summands we have
\begin{align}\begin{split}\label{eq:EcFuncLuis1}
&P(k-\chi) \Big( A(k,\chi) \big(  f^{k-\ell} g^\ell  \big) \Big)= P(k-\chi) \Big( f^{k-\ell} A(k,\chi) \big(   g^\ell  \big) \Big)=P(k-\chi) \big( f^{k-\ell} A(k,\ell)    g^\ell  \big)\\
&=P(k-\ell) \big( f^{k-\ell} A(k,\ell)    g^\ell  \big)= P(k-\ell) \big(  f^{k-\ell} \big)  A(k,\ell)   g^\ell= b(k-\ell)   f^{k-\ell}   A(k,\ell)   g^\ell,
\end{split}\end{align}
just by elementary commuting relations and the fact that $\chi(g^\ell)=\ell g^\ell$. Regarding the second summands in formula \eqref{eq:EcFuncLuis},
\begin{align}\begin{split}\label{eq:EcFuncLuis2}
B(k,\chi)\Big(Q\big(f^{k-\ell}g^\ell\big)\Big)=B(k,\chi)\Big(f^{k-\ell}Q\big(g^\ell\big)\Big)= B(k,\chi)\big(f^{k-\ell}c(\ell)g^\ell\big)=B(k,\ell)f^{k-\ell}c(\ell)g^\ell
\end{split}\end{align}
by the same arguments as above. Putting together formulas \eqref{eq:EcFuncLuis}, \eqref{eq:EcFuncLuis1} and \eqref{eq:EcFuncLuis2} and using the functional equations for $f$ and $g$ and the expression of $(b*c)(s)$, we finally obtain that
$$R(k) \big(h^k\big)=\sum_{\ell=0}^k \binom{k}{\ell}(A(k,\ell)b(k-\ell)+B(k,\ell)c(\ell))f^{k-\ell}g^\ell=(b * c)(k) h^k,$$
hence $R(s)\cdot  h^s =(b * c)(s)\cdot  h^s$.

Now, we know from our hypotheses that $P_j = P_{j0} f + \sum_{r=1}^n P_{jr} f'_{x_r}$, with $P_{jr}\in \cD_{X,x}$ and $Q= \sum_{t=1}^m Q_t g'_{y_t}$, with $Q_t\in\cD_{Y,y}$. Therefore,
\begin{align*}
R(s) &= P(s-\chi) A(s,\chi) + B(s,\chi) Q = A(s,\chi) P(s-\chi) + B(s,\chi) Q \\
&=\sum_j A(s,\chi) (s-\chi)^j P_j + B(s,\chi) Q,
\end{align*}
that belongs to $\cD_{X\times Y,(x,y)}[s] \langle f, f'_{x_1},\dots,  f'_{x_n}, g'_{y_1},\dots , g'_{y_m} \rangle$. However, since $g = \chi(g) \in \cO_{Y,y} \langle g'_{y_1},\dots g'_{y_m} \rangle$, we can affirm that $\langle f,f'_{x_1},\dots,f'_{x_n},g'_{y_1},\dots g'_{y_m} \rangle=\langle f+g,f'_{x_1},\dots,f'_{x_n},g'_{y_1},\dots g'_{y_m} \rangle=J_h$.

The last claim is just an easy consequence of taking $b(s)=\btil_f(s)$ and $c(s)=\btil_g(s)$.
\end{proof}

\begin{exas}
Let $X=\AA_x^1$, $Y=\AA_y^1$, and let us consider the well-known example of the cusp $h=x^2+y^3$, that is, the sum of $f=x^2$ and $g=y^3$. As an example of our main result, we can obtain not just a multiple of the reduced Bernstein-Sato polynomial of $h$ (in fact, the actual polynomial), but also a functional equation. Note in this case that both $f$ and $g$ are evidently Euler-homogeneous. Let us choose $g$ as such, so that $\chi\in\cD_{Y,y}$ will be $\frac{1}{3}y\dd_y$.

On one hand, we have that $\btil_f(s)=(s+1/2)$ and $\btil_g(t)=(t+1/3)(t+2/3)$. In this case, $(\btil_f*\btil_g)(s)=(s+5/6)(s+7/6)=\btil_h(s)$. On the other hand, following the notation of Theorem \ref{thm:Division1}, we can take $P(s)=\frac{1}{2}\dd_xx=\frac{1}{2}(x\dd_x+1)$, $Q=\frac{1}{9}\dd_y^2y^2=\frac{1}{9}(y^2\dd_y^2+4y\dd_y+2)$, $A(s,t)=s+t+3/2$ and $B(s,t)=1$.

Summing up, we have $R(s)\cdot h^s=(\btil_f*\btil_g)(s)\cdot h^s$, where
\begin{align*}
R(s)&=A(s,\chi)P(s-\chi)+B(s,\chi)Q=\left(s+\frac{1}{3}y\dd_y+\frac{3}{2}\right)\frac{1}{2}(x\dd_x+1) +\frac{1}{9}(y^2\dd_y^2+4y\dd_y+2)\\
&=\frac{1}{2}(x\dd_x+1)s+\frac{1}{6}xy\dd_x\dd_y+\frac{1}{9}y^2\dd_y^2+\frac{3}{4}x\dd_x +\frac{11}{18}y\dd_y+\frac{35}{36}.
\end{align*}

\vspace{.1cm}

The example above can obviously be extended to the case of any suspension of the form $h(x_1,\ldots,x_n,z)=z^r+f(x_1,\ldots,x_n)$, for any $f\in\cO_{\AA^n}$ and $r\geq2$. In that case we can take $g(z)=z^r$, for which we know that $\chi=\frac{1}{r}z\dd_z$, $Q=\frac{1}{r^{r-1}}\dd_z^{r-1}z^{r-1}$ and $\btil_g(t)=\prod_{i=1}^{r-1}(t+i/r)$. If we have a reduced Bernstein-Sato functional equation of the form $P(s)\cdot f^s=\btil_f(s)\cdot f^s$, we could write
$$R(s)\cdot h^s=(\btil_f*\btil_g)(s)\cdot h(s),$$
where
$$R(s)=A\left(s,\frac{1}{r}z\dd_z\right)P\left(s-\frac{1}{r}z\dd_z\right)+ B\left(s,\frac{1}{r}z\dd_z\right)\prod_{i=1}^{r-1}(t+i/r),$$
$A(s,t),B(s,t)\in\CC[s,t]$ being such that $(\btil_f*\btil_g)(s)=A(s,t)\btil_f(s-t)+B(s,t)\btil_g(t)$. Note that, for instance, if no pair of roots of $\btil_f$ differ by any $j/r$, with $j=1,\ldots,r$, then $(\btil_f*\btil_g)(s)=\prod_{i=1}^{r-1}\btil_f(s+i/r)$.
\end{exas}

\bibliographystyle{amsplain}
\bibliography{ThomSebastiani}

\providecommand{\bysame}{\leavevmode\hbox to3em{\hrulefill}\thinspace}
\providecommand{\MR}{\relax\ifhmode\unskip\space\fi MR }
\providecommand{\MRhref}[2]{%
  \href{http://www.ams.org/mathscinet-getitem?mr=#1}{#2}
}
\providecommand{\href}[2]{#2}
\begin{thebibliography}{10}

\bibitem{Josepycia}
Josep \`{A}lvarez Montaner, Daniel~J. Hern\'{a}ndez, Jack Jeffries, Luis
  N\'{u}\~{n}ez Betancourt, Pedro Teixeira, and Emily~E. Witt,
  \emph{Bernstein–sato functional equations, {V}-filtrations, and multiplier
  ideals of direct summands}, Commun. Contemp. Math. \textbf{Online Ready}
  (2021), 1--47.

\bibitem{BMS}
Nero Budur, Mircea Musta\c{t}\v{a}, and Morihiko Saito, \emph{Bernstein-{S}ato
  polynomials of arbitrary varieties}, Compos. Math. \textbf{142} (2006),
  no.~3, 779--797.

\bibitem{Granger}
Michel Granger, \emph{Bernstein-{S}ato polynomials and functional equations},
  Algebraic approach to differential equations, World Sci. Publ., Hackensack,
  NJ, 2010, pp.~225--291.

\bibitem{Malgrange}
Bernard Malgrange, \emph{Le polyn\^{o}me de {B}ernstein d'une singularit\'{e}
  isol\'{e}e}, Fourier integral operators and partial differential equations
  ({C}olloq. {I}nternat., {U}niv. {N}ice, {N}ice, 1974), Lecture Notes in
  Math., Vol. 459, Springer, Berlin, 1975, pp.~98--119.

\bibitem{LuisMebkhout}
Zoghman Mebkhout and Luis Narv\'{a}ez-Macarro, \emph{La th\'{e}orie du
  polyn\^{o}me de {B}ernstein-{S}ato pour les alg\`ebres de {T}ate et de
  {D}work-{M}onsky-{W}ashnitzer}, Ann. Sci. \'{E}cole Norm. Sup. (4)
  \textbf{24} (1991), no.~2, 227--256.

\bibitem{Mus}
Mircea Musta\c{t}\u{a}, \emph{Bernstein-{S}ato polynomials for general ideals
  vs. principal ideals}, Proc. Amer. Math. Soc. \textbf{150} (2022), no.~9,
  3655--3662.

\bibitem{Luis}
Luis Narv\'aez~Macarro, \emph{Differential structures in commutative algebra},
  Mini-course at the XXIII Brazilian Algebra Meeting, July 27 - August 1, 2014,
  Maringá, Brazil, 2014.

\bibitem{Nunez}
Luis N\'{u}{\~n}ez-Betancourt, \emph{On certain rings of differentiable type
  and finiteness properties of local cohomology}, J. Algebra \textbf{379}
  (2013), 1--10.

\bibitem{Saito}
Morihiko Saito, \emph{On microlocal {$b$}-function}, Bull. Soc. Math. France
  \textbf{122} (1994), no.~2, 163--184.

\bibitem{Yano}
Tamaki Yano, \emph{On the theory of {$b$}-functions}, Publ. Res. Inst. Math.
  Sci. \textbf{14} (1978), no.~1, 111--202.

\end{thebibliography}

\end{document}